\documentclass{article}
\usepackage{amsmath,amsthm,amssymb,amsfonts,bm,url,fullpage,graphicx}
\usepackage[all,rotate]{xy}
\newtheorem{thm}{Theorem}

\newtheorem{lem}{Lemma}
\theoremstyle{definition}
\newtheorem{Def}{Definition}

\newcommand{\arr}{\ar@{-}}
\renewcommand\1{\bm{1}}
\newcommand\x{\bm{x}}
\newcommand{\C}{\mathcal{C}}
\newcommand{\justarc}{\xymatrix@R1pt@C1pt{&&\\ \ar@{-} `ru/2pt[rr] [rr]&&}}
\newcommand{\rone}{\xymatrix@R.1pt@C.1pt{&&&\\&&&\\&&&\\ %
\ar@{-} `ru/2pt[ruuu] |!{[uu];[rr]}\hole `/2pt[uu] `/2pt[ru] [rr]&& %
}}

\newcommand{\F}{\mathbb{F}}

\newcommand{\sumallx}{\sum_{ %
\xy %
*[*.7]\xybox{ %
*+{a}*\frm{o}, !RU(1.2) *[*.8]{\x}} %
\endxy}}
\newcommand{\merge}[2]{\xy %
*++{} *\frm{o} ="bp", !D *{\bullet}, "bp"!RU(1.2) *[*.8]{#1}, "bp"; /r 30pt/ *++{} *\frm{o} ="other" **\dir{-} , "other"!RU(1.2) *[*.8]{#2} %
\endxy}

\title{The Bar-Natan Theory Splits}
\author{Yuval Wigderson\\Princeton University\\ \url{yuvalw@princeton.edu}}
\date{}

\begin{document}

\maketitle

\begin{abstract}
We show that over the binary field $\F_2$, the Bar-Natan perturbation of Khovanov homology splits as the direct sum of its two reduced theories, which we also prove are isomorphic. This extends Shumakovitch's analogous result for ordinary Khovanov homology, without the perturbation.
\end{abstract}

\section{Introduction}
In \cite{khovanov}, Khovanov introduced what he called the ``categorification of the Jones polynomial'', and which has since become known as Khovanov homology. Khovanov's idea was to construct a link invariant as the homology of a certain chain complex associated to the link, in such a way that the graded Euler characteristic of this homology theory is the link's Jones polynomial. In \cite{bnintro}, Bar-Natan calculated the Khovanov homology of many knots, and found that it is a strictly stronger invariant than the Jones polynomial---there are pairs of knots whose Jones polynomials are equal but whose Khovanov homologies are non-isomorphic. 

However, it is misleading to think of Khovanov homology as a simple extension of the Jones polynomial, as its theory turns out to be very rich. It has been developed and extended in many ways, including the anti-commutative ``odd'' version defined by Osv\'ath, Rasmussen, and Sz\'abo in \cite{odd}, the higher-dimensional analogue introduced by Sz\'abo in \cite{szabo}, and its extensions to tangle invariants defined by Khovanov in \cite{khovtangle} and by Bar-Natan in \cite{bntangle}. This latter paper, \cite{bntangle}, is also important because it is where Bar-Natan introduced a ``perturbation'' to Khovanov homology (similar to Lee's perturbation \cite{lee}) which has become a major area of study.

In this paper, we begin by introducing Khovanov homology and the Bar-Natan perturbation. Following Khovanov's \cite{reduced}, the two reduced Bar-Natan theories are introduced, and we finish by extending a result of Shumakovitch in \cite{shumakovitch} to show that over the binary field $\F_2$, the Bar-Natan theory splits as the direct sum of the two reduced theories. 

\paragraph{Acknowledgements:} I would like to thank Sucharit Sarkar for his extremely helpful comments. This study was conducted while the author was supported for summer research from NSF Grant DMS-1350037.

\section{Khovanov Homology}
Let $\mathcal L$ be a link, and $\F_2$ be the field with two elements\footnote{Everything below also works over any field, and indeed over $\mathbb{Z}$. However, since our main result deals with the $\F_2$ version of Khovanov homology, we won't bother with the other versions.}. In this section, we summarize the definition of the Khovanov homology of $\mathcal L$, as defined in \cite{khovanov}. For this, we begin by fixing an oriented knot diagram $L$ of $\mathcal L$: this is a collection of oriented arcs in the plane, with $n$ double points. We write $n=n_++n_-$, where $n_+$ (resp.\ $n_-$) is the number of positive (resp.\ negative) crossings in $L$. We order these double points arbitrarily, thus identifying them with the set $\{1,2,\ldots,n\}$. Each double point is of the form
$$\xymatrix{
\ar@{-}[rd]&\\
\ar@{-}[ru]|\hole&
}$$
We define two ``resolutions'' of such a crossing:
\begin{center}
\includegraphics{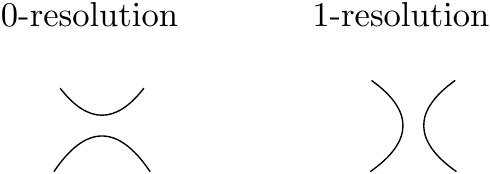} 
\end{center}
So for each vertex $\alpha$ of the hypercube $\{0,1\}^n$, we can resolve each crossing of $L$ in the obvious way: if the $j$th coordinate of $\alpha$ is $0$, then we use the $0$-resolution for crossing $j$, and we use the $1$-resolution otherwise. Note that any such ``full resolution'' can have no crossing points, so it must be a collection of disjoint circles embedded in the plane. We call this collection of circles $D_\alpha(L)$. We further associate to each vertex $\alpha$ a vector space $V_\alpha(L)$ of dimension $2^k$, where $k$ is the number of circles in $D_\alpha$. We identify the basis of $V_\alpha(L)$ with the collection of labellings of the circles in $D_\alpha(L)$ by two symbols, which we call $\1$ and $\x$. We can do this formally by letting $W$ be the $2$-dimensional vector space spanned by $\1$ and $\x$, and then declaring $V_\alpha(L):=W^{\otimes k}$. Moreover, we introduce a grading to elements of $V_\alpha(L)$ by declaring $gr(\1)=1,gr(\x)=-1$ and by extending grading additively to the tensor product $V_\alpha(L)$. Thus, a basis element of $V_\alpha(L)$ is a labelling of the $k$ circles in $D_\alpha(L)$, and its grading is the number of $\1$'s in this labelling minus the number of $\x$'s.

\subsection{Chain Groups}
The chain groups in the chain complex defining Khovanov homology are direct sums of these $V_\alpha$'s: 
\begin{Def}[Khovanov, \cite{khovanov}]
The Khovanov chain complex of $L$ has, as its $i$th chain group,
$$C_i:=\bigoplus_{\alpha \in \{0,1\}^n:w(\alpha)=i+n_-}V_\alpha(L)$$
where $w(\alpha)$ is the Hamming weight of $\alpha$, namely the number of $1$'s in $\alpha$. Each vector space $C_i$ is endowed with a \emph{quantum grading} defined by
$$q(v)=gr(v)+i+n_+-n_-$$
Since there are no hypercube vertices $\alpha$ with $w(\alpha)<0$ or $w(\alpha)>n$, we see that all chain groups $C_i$ with $i<-n_-$ or $i>n_+$ will be the $0$ group. 
\end{Def}
The Khovanov homology of $\mathcal L$ will be defined as the homology\footnote{Under the above definition, the differential will be a map $C_i \to C_{i+1}$, so Khovanov homology will, strictly speaking, be a cohomology theory.} of this chain complex, so we need to define a differential on this complex. 

\subsection{The Differential}
Defining a differential $d_i:C_i \to C_{i+1}$ will be done in several steps, beginning by defining a map for each edge of the hypercube. Note that an edge of the hypercube connects two vertices $\alpha$ and $\beta$ which necessarily differ in exactly one coordinate, say the $j$th. Without loss of generality, $\alpha$ has a $0$ in the $j$th coordinate, while $\beta$ will have a $1$ there. This means that the only difference between $D_\alpha(L)$ and $D_\beta(L)$ is in the resolution of the $j$th crossing, where $D_\alpha$ uses a $0$-resolution, while $D_\beta$ uses a $1$-resolution. In order to mark how changing a $0$-resolution to a $1$-resolution will affect the diagram, we place a little arc in each $0$-resolution, as follows:
\begin{center}
\includegraphics{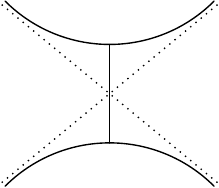}
\end{center}

As shown, we place this arc so that it overlaps with the original crossing point. Then, going from a $0$-resolution to a $1$-resolution simply involves performing surgery along this arc: we translate the arc a bit to the left and a bit to the right and treat these new segments as part of the $1$-resolution.
\begin{center}
\includegraphics{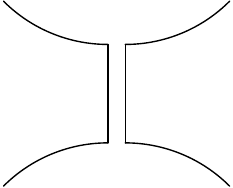}
\end{center}

We now consider the two possible cases:
\paragraph{Case 1:} The two arcs in the $0$-resolution of the $j$th crossing belong to two distinct circles of $D_\alpha(L)$. In this case, moving from a $0$-resolution to a $1$-resolution turns these two circles into one, while all other circles remain unchanged. This situation is called a \emph{merge}.
\paragraph{Case 2:} The two arcs in the $0$-resolution of the $j$th crossing belong to the same circle. In this case, going from a $0$-resolution to a $1$-resolution will turn this single circle into two distinct circles, while keeping all other circles unchanged. This situation is called a \emph{split}.

Using the above notation, we see that the following are a merge and a split, respectively:
\begin{center}
\includegraphics{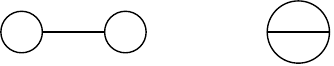}
\end{center}

Given such an edge between $\alpha$ and $\beta$, we wish to define a map $d_{\alpha,\beta}:V_\alpha(L) \to V_\beta(L)$. Of course, it suffices to define $d_{\alpha,\beta}$ on the basis of $V_\alpha(L)$. Recall that a basis vector for $V_\alpha(L)$ is simply a labelling of the circles in $D_\alpha(L)$, and similarly for $V_\beta(L)$. Given such a basis vector, we wish to define its image under $d_{\alpha,\beta}$. By the above, we know that most of the circles in $D_\alpha(L)$ are unchanged by the re-resolution of the $j$th crossing, so for these circles, we can simply define the labelling of the corresponding circle in $D_\beta(L)$ to be its original labelling in our given basis vector. 

Now, if the re-resolution is a split, we have taken into account the original labelling of all but one circle (namely the splitting circle), and we have defined the labellings of all but two target circles (namely the two daughter circles). So our task boils down to defining a map $\Delta:W \to W \otimes W$. We do this by declaring
$$\begin{array}{ccc}
\Delta(\1)=\1 \otimes \x+\x \otimes \1&&\Delta(\x)=\x \otimes \x
\end{array}$$
Similarly, in the case of a merge, our task is simply to define a map $m:W \otimes W \to W$, which we do by declaring
$$\begin{array}{ccc}
m(\1 \otimes \1)=\1 &&m(\x \otimes \x)=0  \\
m(\1\otimes \x)=\x && m(\x \otimes \1)=\x 
\end{array}$$
Then, in addition to the above identity actions on the circles not involved in the re-resolution, we have defined a map $d_{\alpha,\beta}:V_\alpha(L) \to V_\beta(L)$. For convenience, we also define $d_{\alpha,\beta}$ to be the $0$ map in case there is no edge between $\alpha$ and $\beta$. These maps satisfy a commutativity relation, namely 
\begin{lem}[Khovanov, \cite{khovanov}] \label{squarecomm}
For each $2$-dimensional face in the hypercube, the maps on the edges commute. More formally, a $2$-dimensional face has vertices $\{\alpha,\beta_1,\beta_2,\gamma\}$, where $\beta_1$ and $\beta_2$ are gotten from $\alpha$ by changing one $0$ coordinate to $1$, and $\gamma$ is similarly gotten from both $\beta_1$ and $\beta_2$. Then for all such faces,
$$d_{\beta_1,\gamma} \circ d_{\alpha,\beta_1}=d_{\beta_2,\gamma} \circ d_{\alpha,\beta_2}$$
\end{lem}
We then define the differential as follows.
\begin{Def}[Khovanov, \cite{khovanov}]\label{differential}
A map $d_i:C_i \to C_{i+1}$ is defined by
$$d_i:=\bigoplus_{\alpha,\beta~:~w(\alpha)=i+n_-~,~w(\beta)=i+1-n_-} d_{\alpha,\beta}$$
By Lemma \ref{squarecomm}, the maps $d_{\alpha,\beta}$ commute on each $2$-dimensional face, which is the same as anticommuting since we are working over $\F_2$. Therefore, $d_{i+1} \circ d_i=0$ for all $i$, which means that $d$ is a differential\footnote{For the general Khovanov homology (not over $\F_2$), it is necessary to add some minus signs to this definition. }.
\end{Def}
\begin{Def}[Khovanov, \cite{khovanov}]
The Khovanov chain complex $\C(L)$ is the chain complex $(C_*,d_*)$, with chain groups $C_i$ and differential $d_i$ for all $i$. By Lemma \ref{squarecomm} and the above remarks, $\C(L)$ is indeed a chain complex.

The Khovanov homology $Kh_i(L)$ is defined as the $i$th homology group of $\C(L)$. Since the maps $\Delta$ and $m$ both decrease the grading $gr$ by $1$, we can see that they will act homogeneously on the quantum grading, which means that $q$ will descend to a grading on the homology $Kh_i(L)$. Thus, $Kh_i(L)$ is a graded vector space.
\end{Def}

It is clear that the Khovanov chain complex itself is highly dependent on which knot diagram $L$ of the link $\mathcal L$ we pick; for instance, the number of chain groups in $\C(L)$ is precisely the number of crossings in $L$, and there are knot diagrams of the same link with arbitrarily many crossings. Nonetheless:
\begin{thm}[Khovanov, \cite{khovanov}]\label{invariance}
Khovanov homology is a knot invariant. More formally, if $L$ and $L'$ are two knot diagrams of the same link $\mathcal L$, then $Kh_i(L) \cong Kh_i(L')$ for all $i$, where $\cong$ denotes isomorphism of graded vector spaces. 
\end{thm}

\subsection{The Bar-Natan Perturbation}
Now, we define the Bar-Natan perturbation to Khovanov homology. We begin with a purely algebraic formality:
\begin{Def}
Let $H$ be a formal variable, and let $V$ be an $\F_2$-vector space. We extend $V$ to $V[H]$, which is a module over the polynomial ring $\F_2[H]$, as follows. The elements of $V[H]$ are formal linear combinations
\[\sum_{i \geq 0} H^i v_i\]
where $v_i \in V$, and only finitely many summands are nonzero. Addition is defined componentwise, and the action of $\F_2[H]$ is given, naturally, by
\[\lambda \left(\sum_{i \geq 0} H^i v_i\right)=\sum_{i \geq 0} H^i \lambda v_i~~~~~~~~~H\left(\sum_{i \geq 0} H^i v_i\right)=\sum_{i \geq 0} H^{i+1}v_i\]
for $\lambda \in \F_2$. 
\end{Def}
In \cite{bntangle}, Bar-Natan introduced the map $\Delta':W \to W\otimes W$ defined by
\[\Delta'(\1)=\1\otimes\1\]
and $m':W \otimes W \to W$ defined by
\[m'(\x\otimes\x)=\x\]
with all other basis vectors sent to $0$ in either case\footnote{Over fields of characteristic $\neq 2$, it is necessary to add a minus sign to the definition of $\Delta'$.}. From these, he formed edge maps $h_{\alpha,\beta}$, which are defined as $\Delta'$ if the edge $\alpha,\beta$ is a split, and $m'$ if this edge is a merge. Finally, in analogy with Definition \ref{differential}, he defined a homomorphism $h_i:C_i \to C_{i+1}$ by
\[h_i:=\bigoplus_{\alpha,\beta:w(\alpha)=i+n_-\,,\,w(\beta)=i+1-n_-} h_{\alpha,\beta}\]
\begin{Def}[Bar-Natan, \cite{bntangle}]
Let $H$ be a formal parameter of degree $-2$. The Bar-Natan perturbation to the Khovanov chain complex (often called simply the Bar-Natan complex) is the chain complex $\C_{BN}(L)=(C_*[H],d_*+Hh_*)$. In analogy with Lemma \ref{squarecomm}, it is straightforward to check that the edge maps commute on all $2$-dimensional faces of $\{0,1\}^n$, so $\C_{BN}(L)$ is indeed a chain complex. Its $i$th homology group, which is a graded $\F_2[H]$-module, is denoted by $BN_i(L)$.
\end{Def}
He also proved the following theorem:
\begin{thm}[Bar-Natan, \cite{bntangle}]
The Bar-Natan homology is a link invariant. More formally, if $L$ and $L'$ are two knot diagrams of the same link $\mathcal L$, then $BN_i(L) \cong BN_i(L')$ for all $i$, where $\cong$ denotes isomorphism of graded $\F_2[H]$-modules.
\end{thm}

\section{The Bar-Natan Theory Splits}\label{k0}
We pick a basepoint on $L$ and maintain its position in each complete resolution. Let $\C_x$ and $\C_1$ denote the subgroups of $\C_{BN}$ in which the circle with the basepoint is labelled $\x$ and $\1$, respectively. Since both $d$ and $h$ map $\C_x$ into itself, we see that $\C_x$ is actually a subcomplex of $\C_{BN}$, while the quotient $\C_{BN}/\C_x$ is naturally isomorphic (as a group) to $\C_1$. Therefore, we can consider $\C_1$ to be a complex as well, and its differential is given by composing $d+Hh$ with the projection $\C_{BN} \twoheadrightarrow \C_{BN}/\C_x \cong \C_1$. 

Thus, we have a short exact sequence of chain complexes:
\[\xymatrix{0\ar[r]&\C_x \ar[r]&\C_{BN} \ar[r]&\C_1 \ar[r]&0}\]
This induces a long exact sequence on homology:
\[\xymatrix{\dotsb \ar[r]^-\partial&H(\C_x) \ar[r]& H(\C_{BN}) \ar[r]&H(\C_1) \ar[r]^-\partial &H(\C_x) \ar[r]&\dotsb}\]
The homologies $H(\C_x)$ and $H(\C_1)$ are called the reduced Bar-Natan theories. The connecting map $\partial$ is surprisingly easy to understand in this context, since it happens to be induced from a chain map $f:\C_1 \to \C_x$. $f$ acts by applying the differential and then discarding all terms in which the basepoint circle is not labelled $x$; i.e.\ $f=\pi_x \circ (d+Hh)$, where $\pi_x$ is a projection onto $\C_x$. Since the Bar-Natan perturbation $h$ never turns a circle labelled $\1$ into a circle labelled $\x$, it is clear that this connecting map is actually the same as the connecting map gotten in the standard Khovanov case. In particular, $f$ preserves both homological and quantum gradings. 

In \cite{shumakovitch}, Shumakovitch proved that Khovanov homology spits as the direct sum of these two reduced theories, and we wish to extend this result to the case of the Bar-Natan perturbation. More precisely, we prove the following theorem:
\begin{thm}
$H(\C_{BN})=H(\C_x) \oplus H(\C_1)$ when considered as modules over $\F_2[H]$.
\end{thm}
The rest of this section is devoted to proving this theorem. We will do this by proving that the map $f$ is nullhomotopic; this suffices since $\C_{BN}$ is the mapping cone of $f:\C_x \to \C_1$, so if $f$ is nullhomotpic we can conclude that $\C_{BN}$ is chain homotopy equivalent to $\C_x \oplus \C_1$ over $\F_2[H]$. Thus, we will get that $H(\C_{BN})=H(\C_x) \oplus H(\C_1)$ over $\F_2[H]$.
A nullhomotopy for $f$ is simply a map $K$ so that $f=[K,d+Hh]$, where $[\varphi,\psi]$ denotes the commutator $\varphi \circ \psi+\psi \circ \varphi$.\footnote{Since we are working over $\F_2$, the commutator and anticommutator are the same.} $K$ need not be homogeneous, but we can write $K=K_0+HK_1+H^2K_2+\dotsb$, where each $K_i$ is a map that fixes homological grading but raises quantum grading by $2i$. Thus, our goal is now to construct this family of functions.

\subsubsection*{The Khovanov Case: Constructing $K_0$}
The following construction is due to Shumakovitch in \cite{shumakovitch}. 

For a complete resolution $D$ in which the basepoint circle is labelled $\1$, we define
\begin{center}
\includegraphics{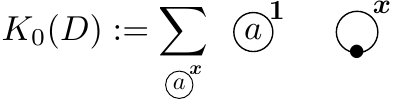}
\end{center}
The sum runs over all circles $a$ labelled $\x$ in $D$. For each such circle, it adds a copy of $D$ in which $a$ has been relabelled to $\1$ and the basepoint circle has been relabelled to $\x$. If this sum is empty or if the basepoint circle is labelled $\x$, then $K_0$ is identically $0$.\footnote{This last condition is another way of saying that $K_0$ is a map $\C_1 \to \C_x$. For if the basepoint circle is labelled $\x$, then the configuration does not lie in the domain of $K_0$, so we declare that $K_0$ evaluates to $0$.} This map clearly preserves homological and quantum grading. We can also check directly that $f=[K_0,d]$, which is the statement that the ordinary Khovanov theory splits over $\F_2$. To see this, we simply check that this equation holds for each 1-dimensional configuration.

We begin with those 1-dimensional configurations involving the basepoint circle. Starting with split maps, we first recall that
\begin{center}
\includegraphics{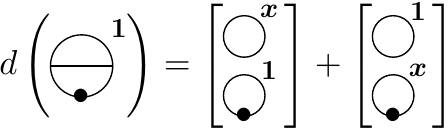}
\end{center}
from which we immediately see that
\begin{center}
\includegraphics{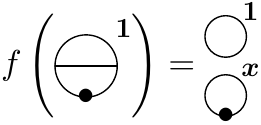}
\end{center}
We can also immediately calculate from this and from the definition of $K_0$ that
\begin{center}
\includegraphics{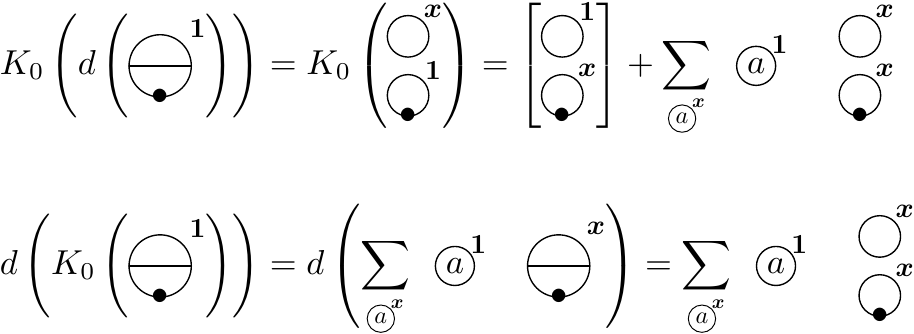}
\end{center}
We see that the $\sumallx$ terms will precisely cancel each other out, giving us
\begin{center}
\includegraphics{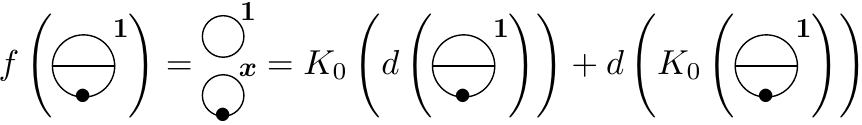}
\end{center}

Next, we consider the other split map, 
\begin{center}
\includegraphics{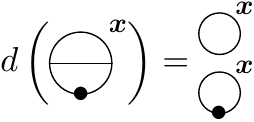}
\end{center}
Since the basepoint circle is labelled $\x$, both $f$ and $K_0$ evaluate to $0$ on this configuration. Moreover, since $d\left(\xy
*++{} *\frm{o} ="bp", "bp"!D *{\bullet},
"bp"!L *{}; "bp"!R *{} **\dir{-} ,
"bp"!RU(1.2)*[*.8]{\x}
\endxy\right)$ has the basepoint labelled $\x$, we see that $K_0\left(d\left(\xy
*++{} *\frm{o} ="bp", "bp"!D *{\bullet},
"bp"!L *{}; "bp"!R *{} **\dir{-} ,
"bp"!RU(1.2)*[*.8]{\x}
\endxy\right)\right )=0$ as well. Thus, $f=[K_0,d]$ on this configuration too.

Now we turn to merge maps, which are as straightforward. We can immediately dispense with the $\merge \x\x$ merge, since $f,d,$ and $K_0$ all evaluate to $0$ on it. Next, we recall that
\begin{center}
\includegraphics{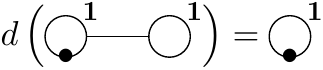}
\end{center}
and therefore
\begin{center}
\includegraphics{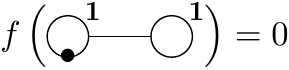}
\end{center}
We then calculate that
\begin{center}
\includegraphics{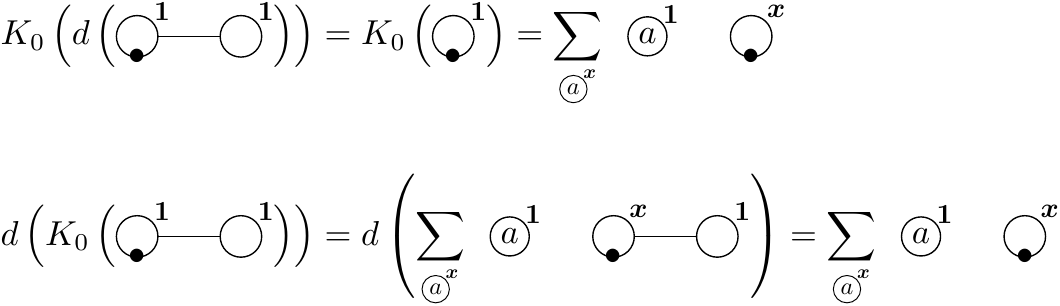}
\end{center}
which shows that $f=[K_0,d]$ for this configuration.

Next, we check the $\merge \1\x$ merge. We have
\begin{center}
\includegraphics{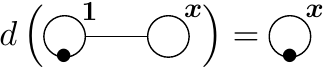}
\end{center}
and therefore
\begin{center}
\includegraphics{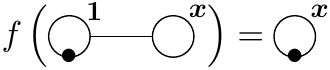}
\end{center}
as well. We also see that
\begin{center}
\includegraphics{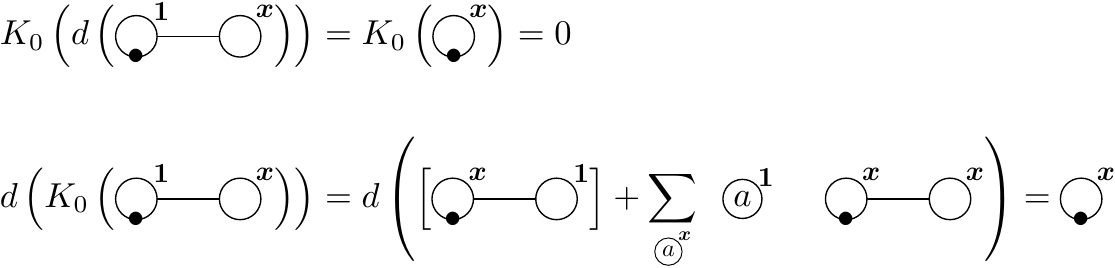}
\end{center}
Here, we can again see that $f=[K_0,d]$ for this configuration. 

Finally, we turn to the $\merge \x\1$ merge. We again have
\begin{center}
\includegraphics{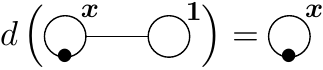}
\end{center}
while $f\left(\merge \x\1\right)=0$. In addition,
\begin{center}
\includegraphics{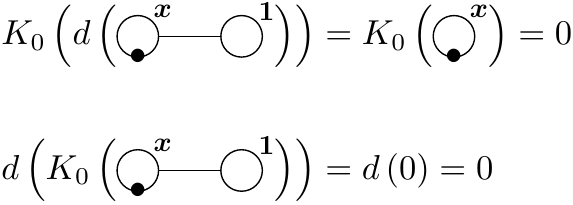}
\end{center}

Now, we must check that $f=[K_0,d]$ also for 1-dimensional configurations that don't involve the basepoint circle. Note that on all of these configurations, $f$ will be $0$, since surgery along any edge not connected to the basepoint circle cannot turn the basepoint circle from a $\1$ to an $\x$. Therefore, we wish to prove that $[K_0,d]=0$. Starting from the splits, we have
\begin{center}
\includegraphics{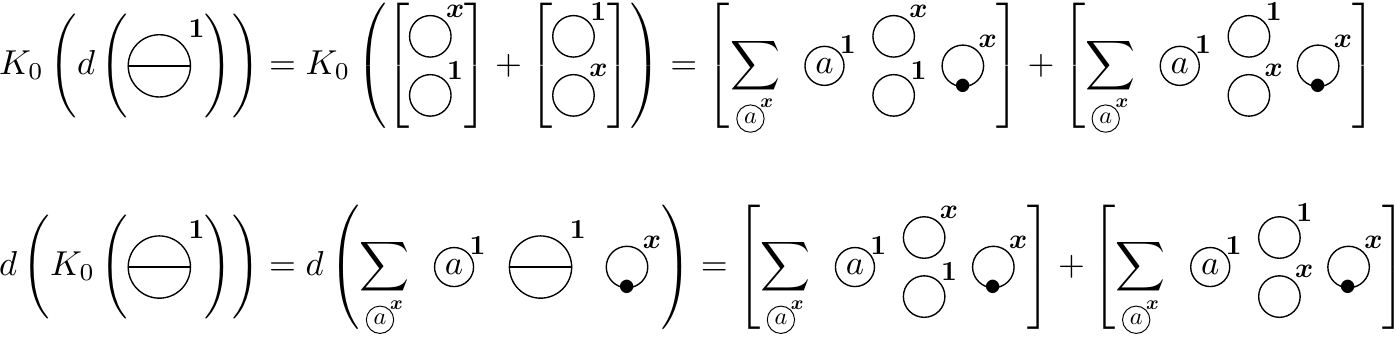}
\end{center}
Note that in calculating $K_0 \circ d$, we also get two terms in which both daughter circles are labelled $\1$, but these two terms cancel each other out since we are working over $\mathbb F_2$. 

Next, we consider the $\x$ split:
\begin{center}
\includegraphics{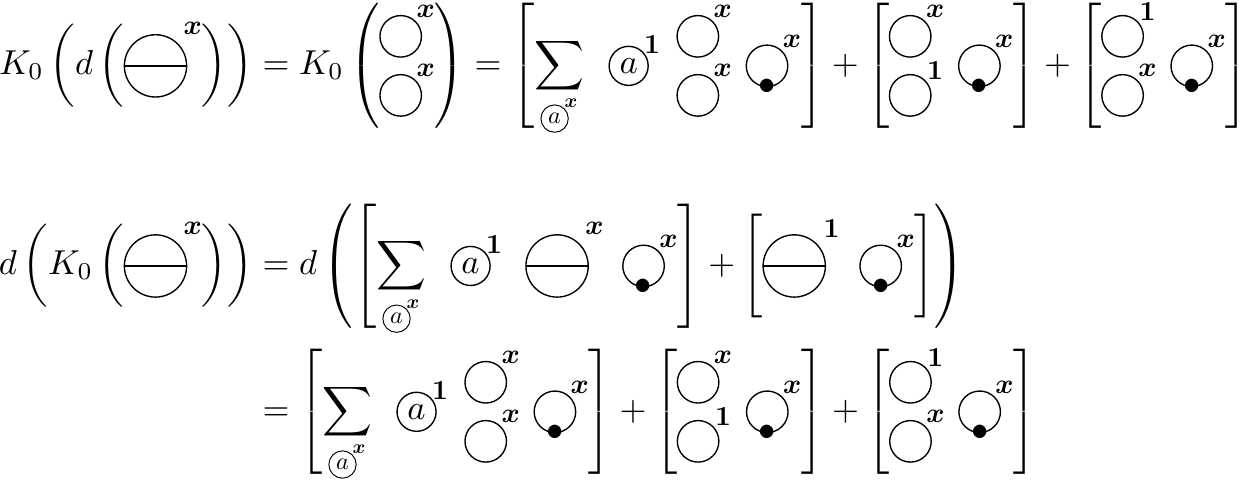}
\end{center}
and we see that these are indeed equal.

Next, we turn to the merges. We have
\begin{center}
\includegraphics{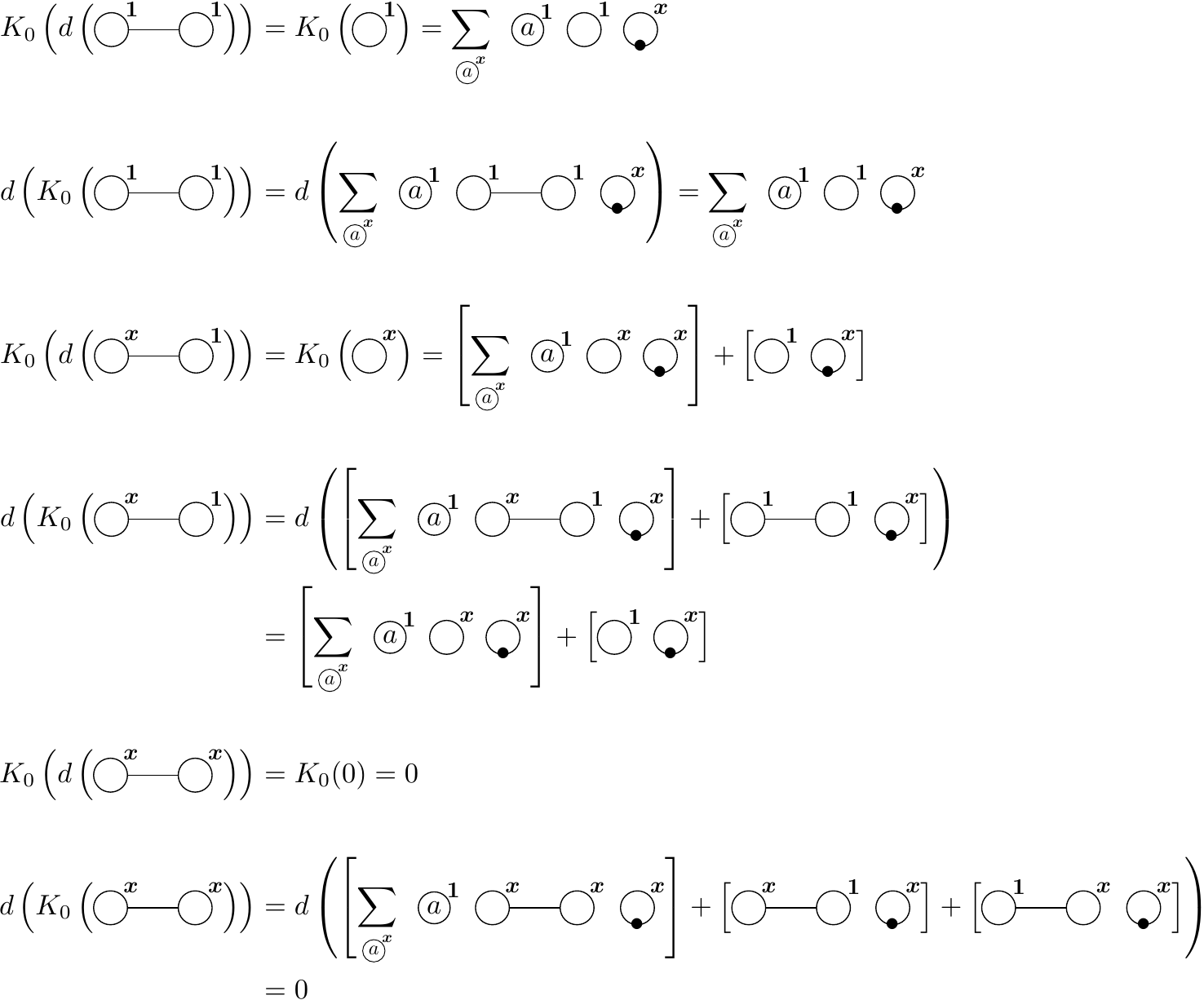}
\end{center}
Thus, since each pair sums to 0, we see that $f=[K_0,d]$.

\subsubsection*{The Bar-Natan Case: The Higher $K_i$'s}
In order to see that the Bar-Natan theory splits as well, we will construct the homotopies $K_i$ for all $i>0$. Recall that $K_i$ must preserve homological grading but raise quantum grading by $2i$, so there is a natural guess for a definition of $K_i$:
\begin{center}
\includegraphics{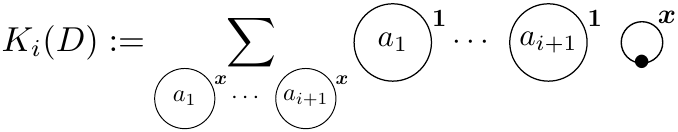}
\end{center}
This sum runs over all $(i+1)$-tuples of circles labelled $\x$. As usual, if the basepoint circle is labelled $\x$ in $D$ or if this sum is empty, then $K_i(D)=0$. Note that when $i=0$, this actually agrees with our above definition of $K_0$. Since we intend to define $K=K_0+HK_1+H^2 K_2+\dotsb$ and we intend to have $f=[d+Hh,K]$, we can see that we need
\[[K_i,h]+[K_{i+1},d]=0\]
to be satisfied for all $i>0$. We will begin by proving that this holds for all $1$-dimensional configurations involving the basepoint.

We can immediately see that $[K_i,h]$ has only two nonzero terms involving the basepoint, namely:
\begin{center}
\includegraphics{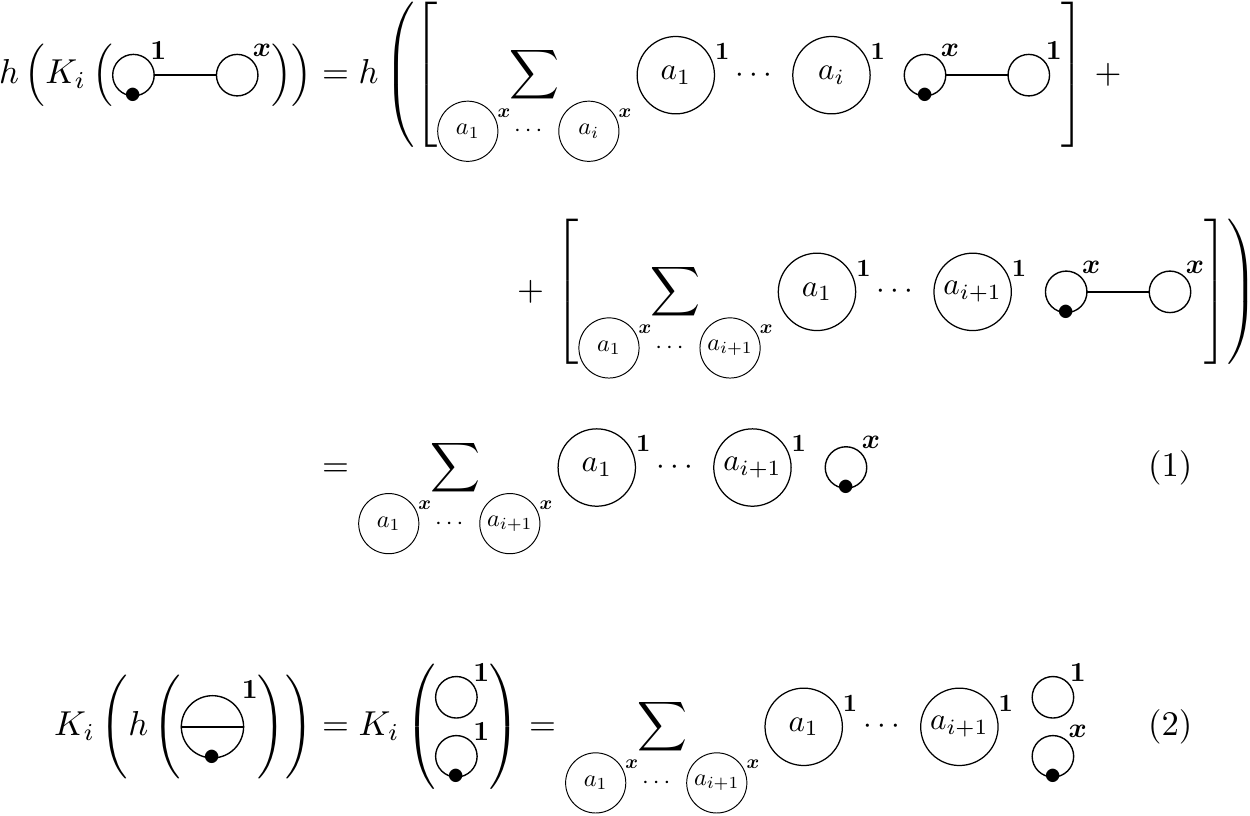}
\end{center}
We need to check that $[K_{i+1},d]$ has the same nonzero terms containing the basepoint, and we'll do this by first checking that $[K_{i+1},d]$ agrees with $[K_i,h]$ on these configurations, and then that it has no other nonzero basepoint terms. First, we note that $K_{i+1}\left(d\left(\merge \1\x\right)\right)=0$, since $d\left(\merge \1\x\right)$ has the basepoint circle labelled $\x$. On the other hand,
\begin{center}
\includegraphics{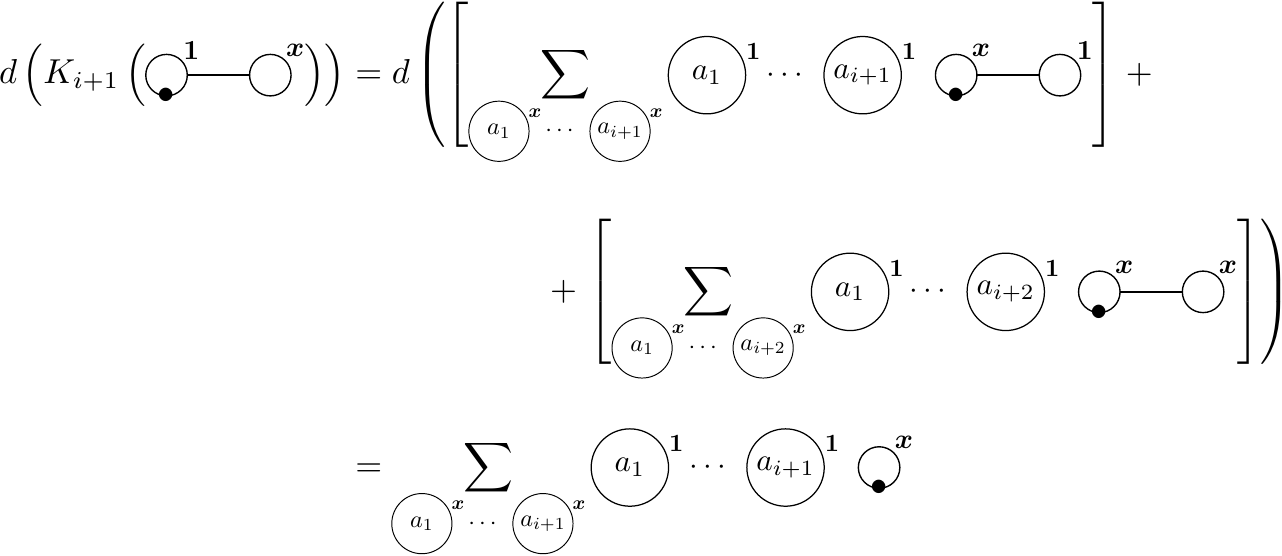}
\end{center}
which agrees with (1).

For the split, we have
\begin{center}
\includegraphics{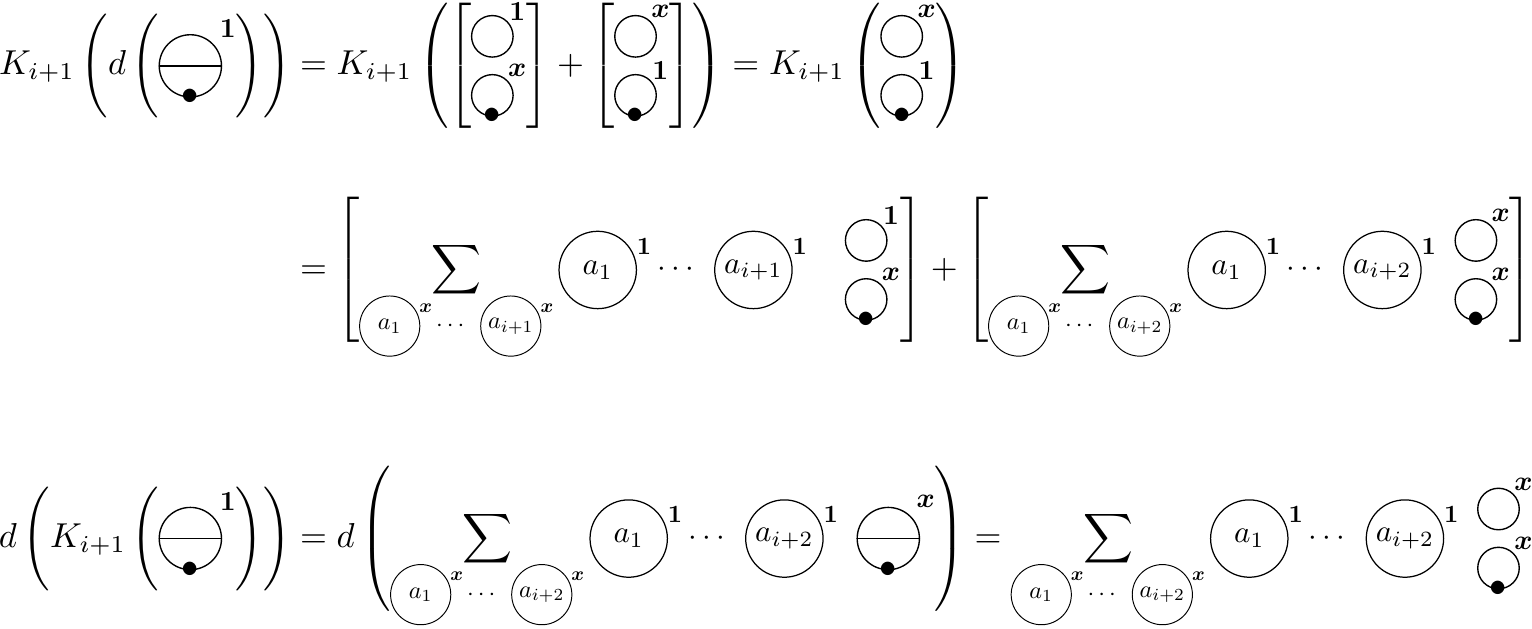}
\end{center}
And indeed, their sum precisely equals (2). 

Now, we check the other configurations involving the basepoint. We can immediately ignore the $\x,\x$ merge, since $K_{i+1}$ and $d$ both vanish on it. Similarly, we can dispense with the $\x$ split and the $\x,\1$ merge, since $K_{i+1}$ will vanish on both the input and the output. All that remains, therefore, is the $\1,\1$ merge. For this, we have
\begin{center}
\includegraphics{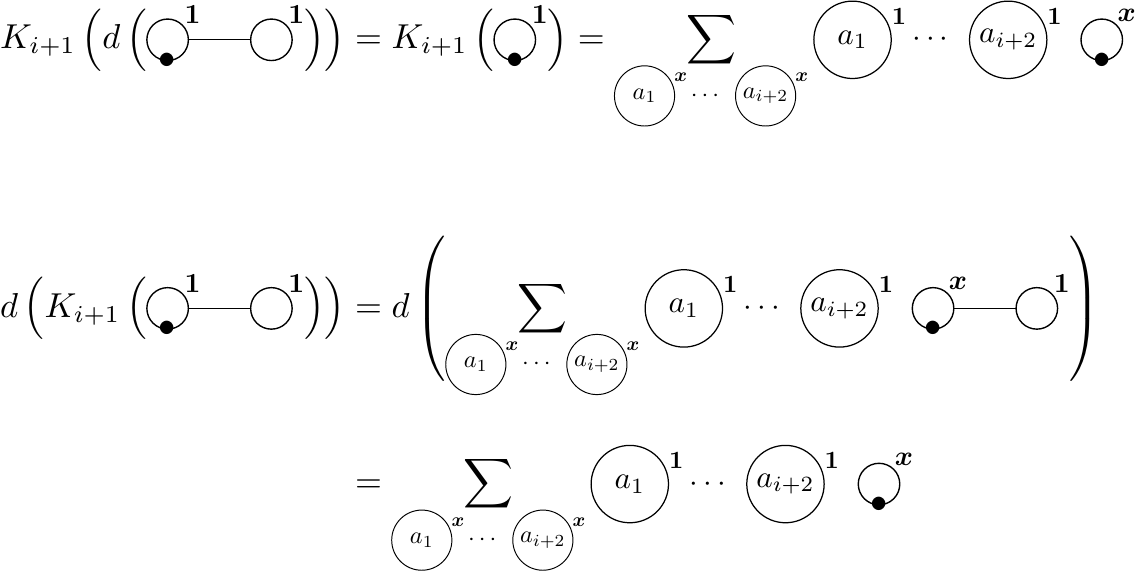}
\end{center}
So we indeed see that $[K_{i+1},d]$ has no other nonzero terms involving the basepoint.

Next, we need to check that $[K_i,h]=[K_{i+1},d]$ on all $1$-dimensional configurations not involving the basepoint. We begin by calculating $[K_i,h]$ on all such configurations, starting with the splits:
\begin{center}
\includegraphics{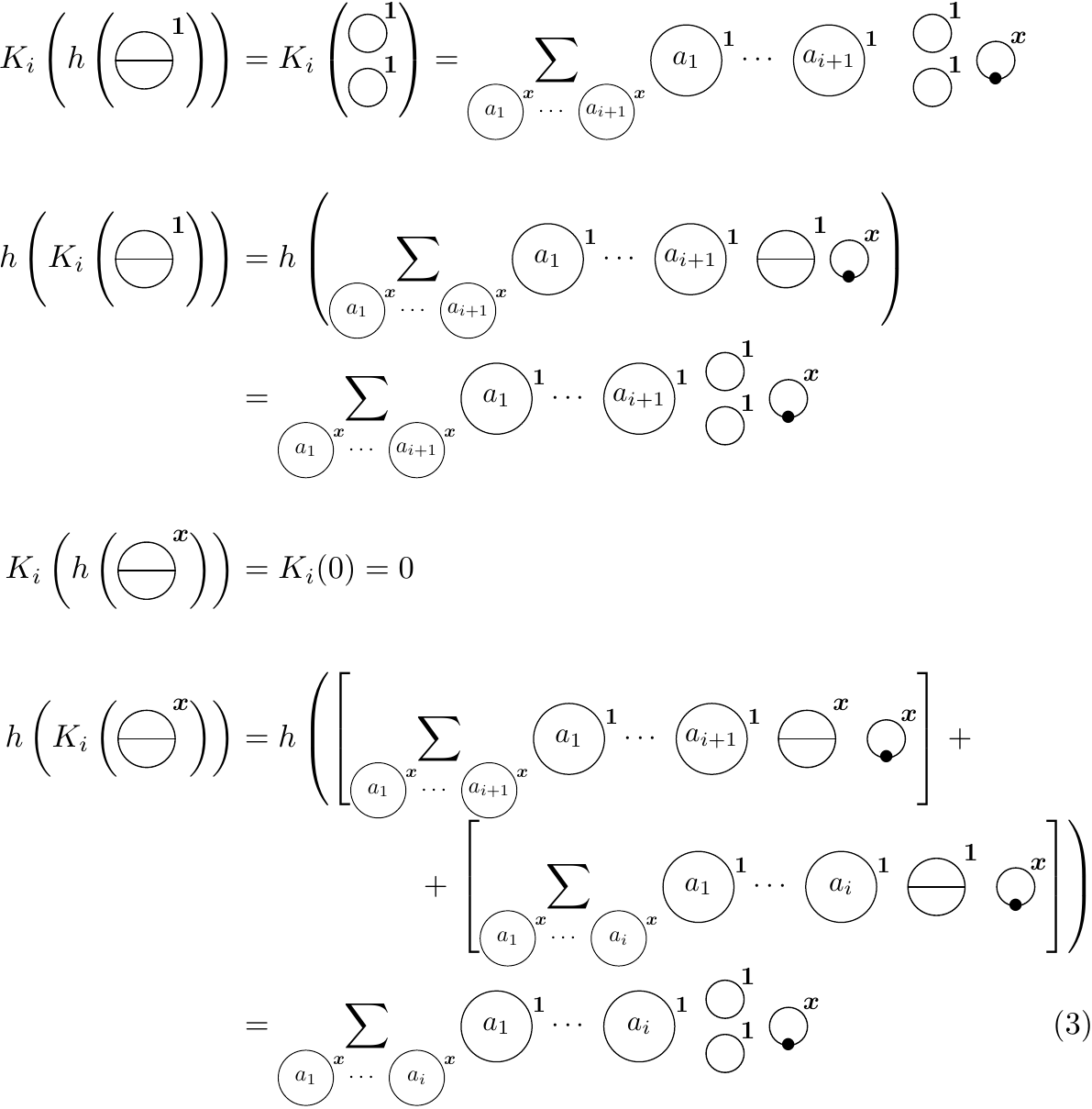}
\end{center}
Thus, the only nonzero term we need to worry about here is (3), which will be precisely cancelled soon.

Next, consider $[K_i,h]$ on the merges.
\begin{center}
\includegraphics{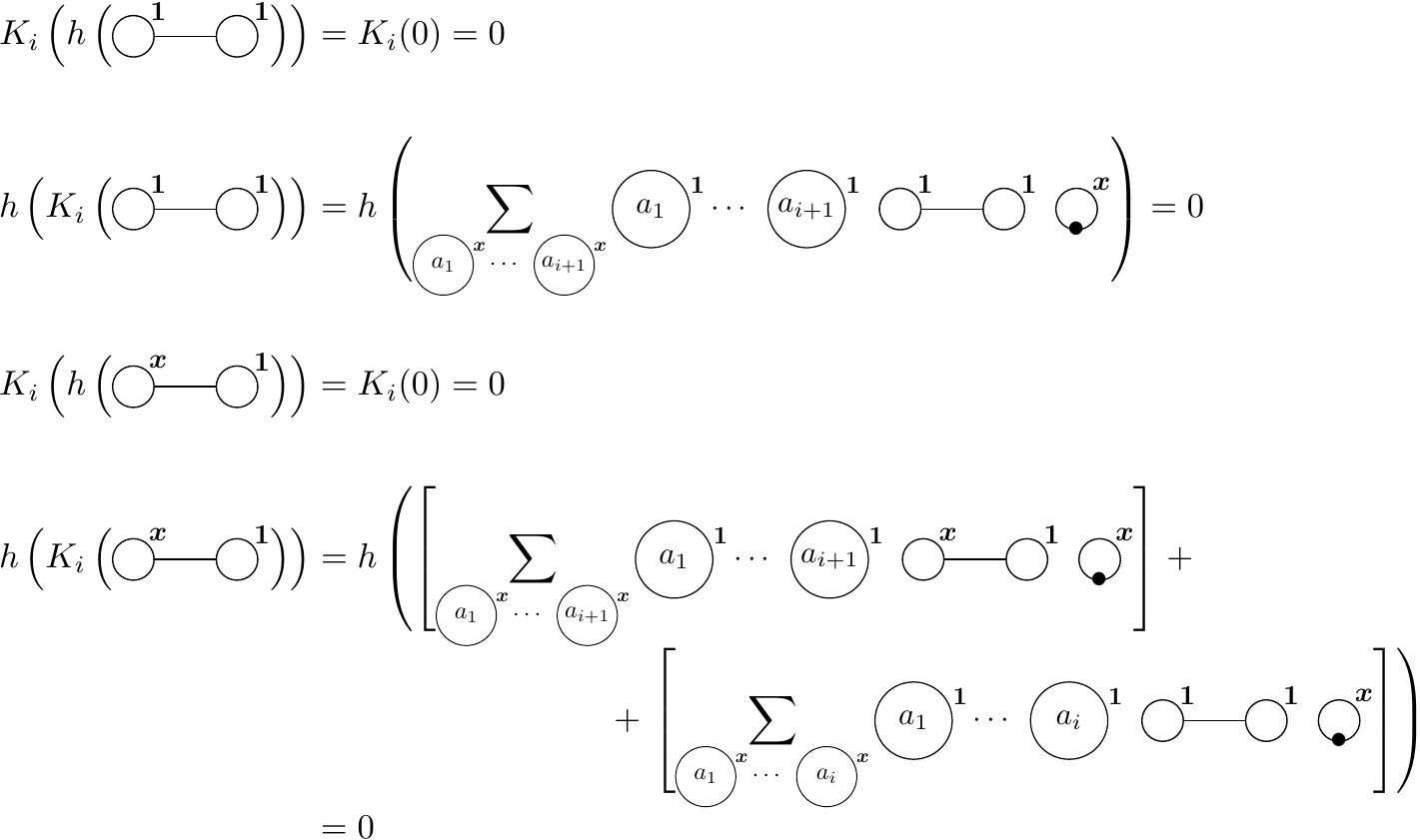}
\end{center}
\begin{center}
\includegraphics{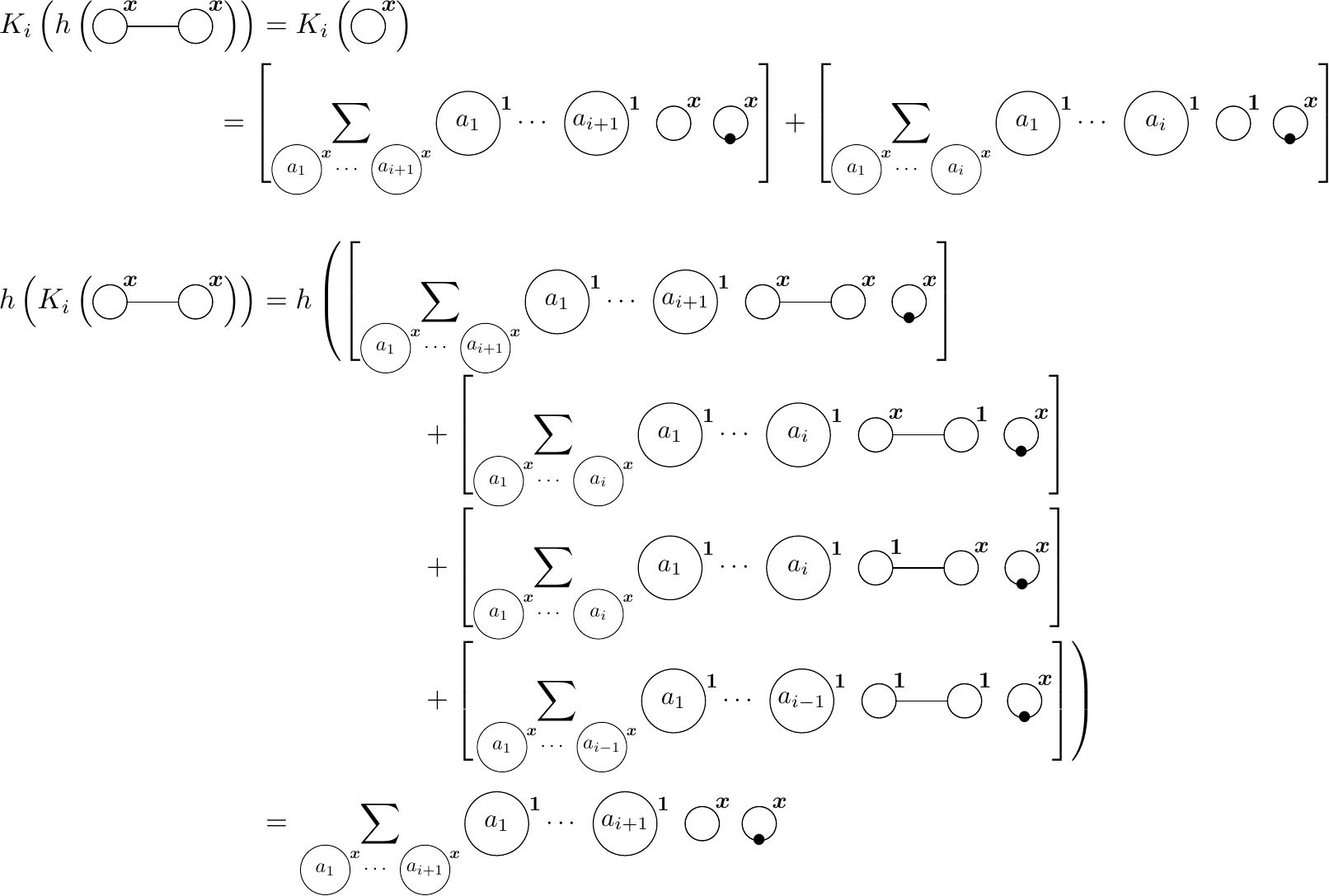}
\end{center}
Thus, the only nonzero term we need to worry about here is
\begin{center}
\includegraphics{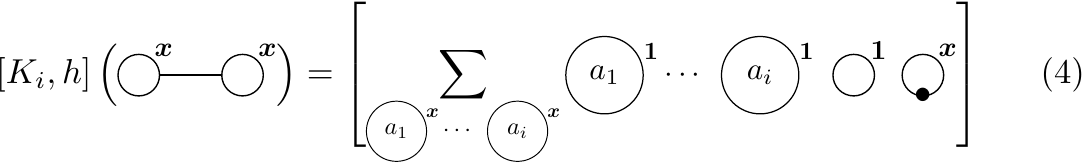}
\end{center}

Next, we do the same for $[K_{i+1},d]$, again beginning with the splits:
\begin{center}
\includegraphics{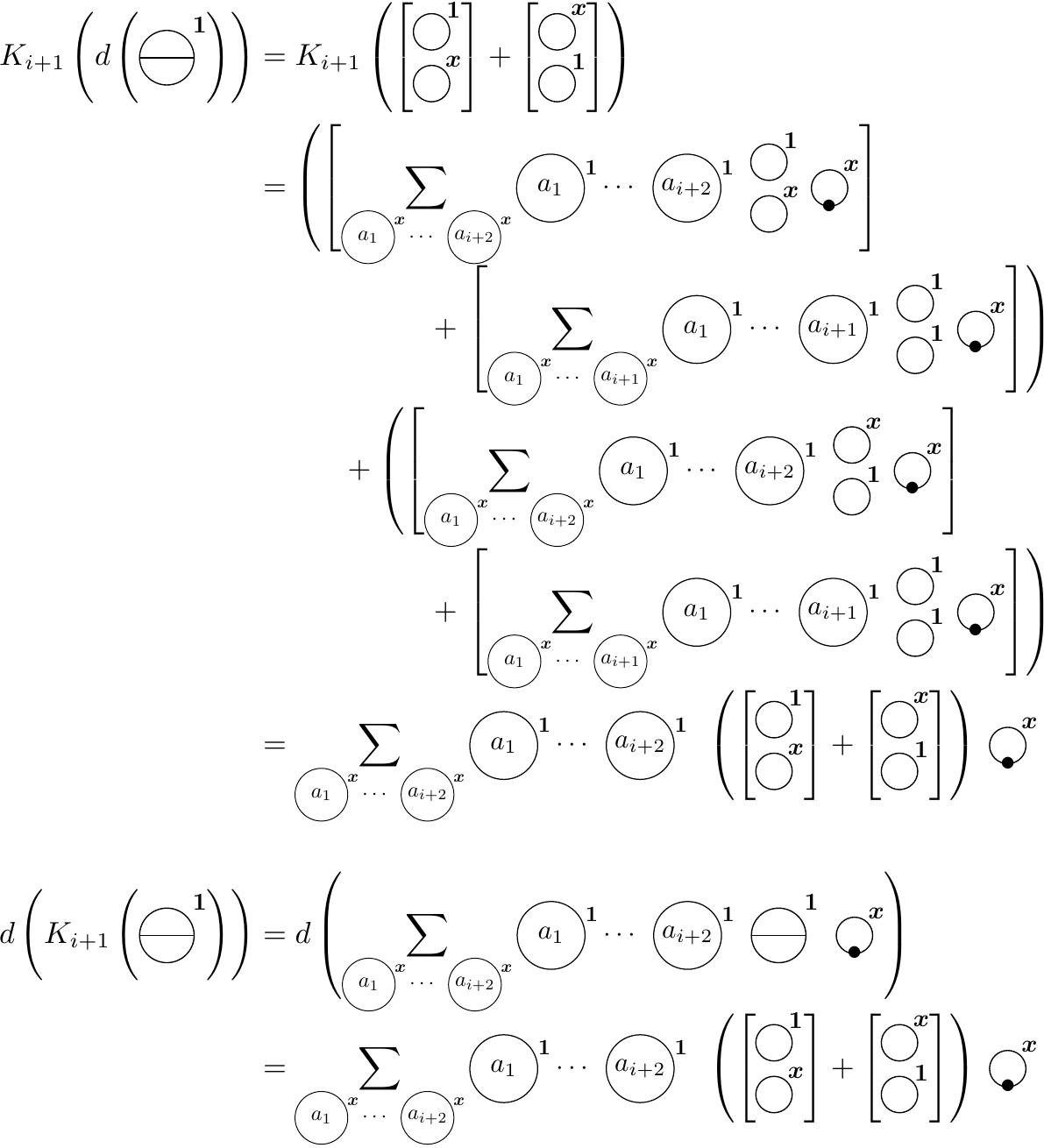}
\end{center}
\begin{center}
\includegraphics{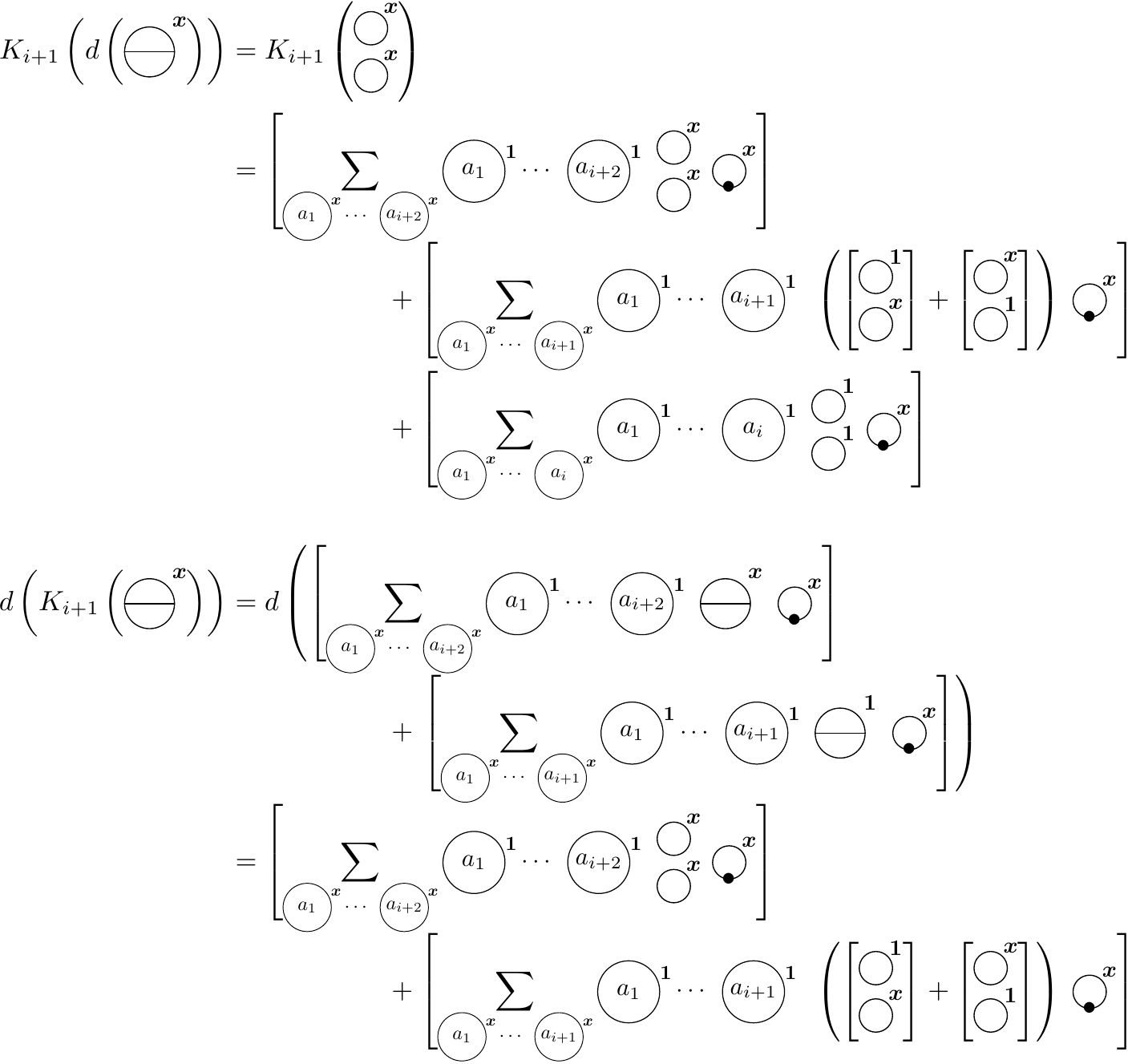}
\end{center}
Therefore, the only term that remains from the splits is
\begin{center}
\includegraphics{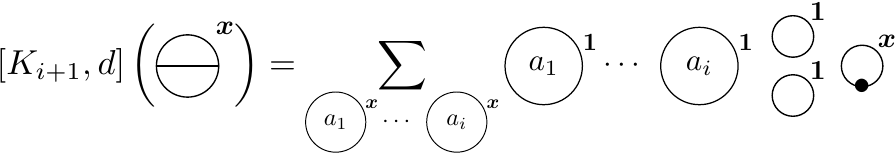}
\end{center}
which precisely equals (3).

Next, we check the merges:
\begin{center}
\includegraphics{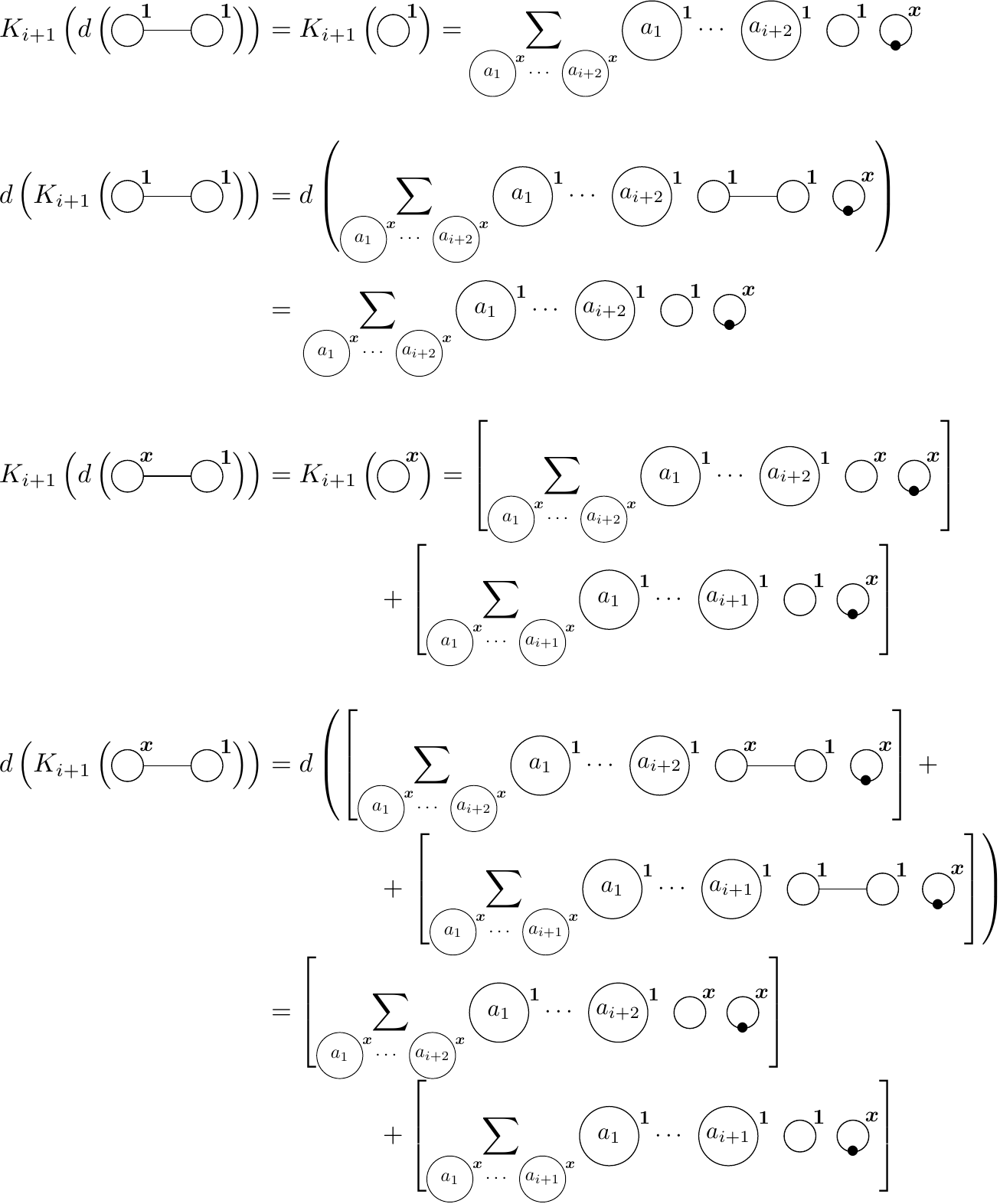}
\end{center}
\begin{center}
\includegraphics{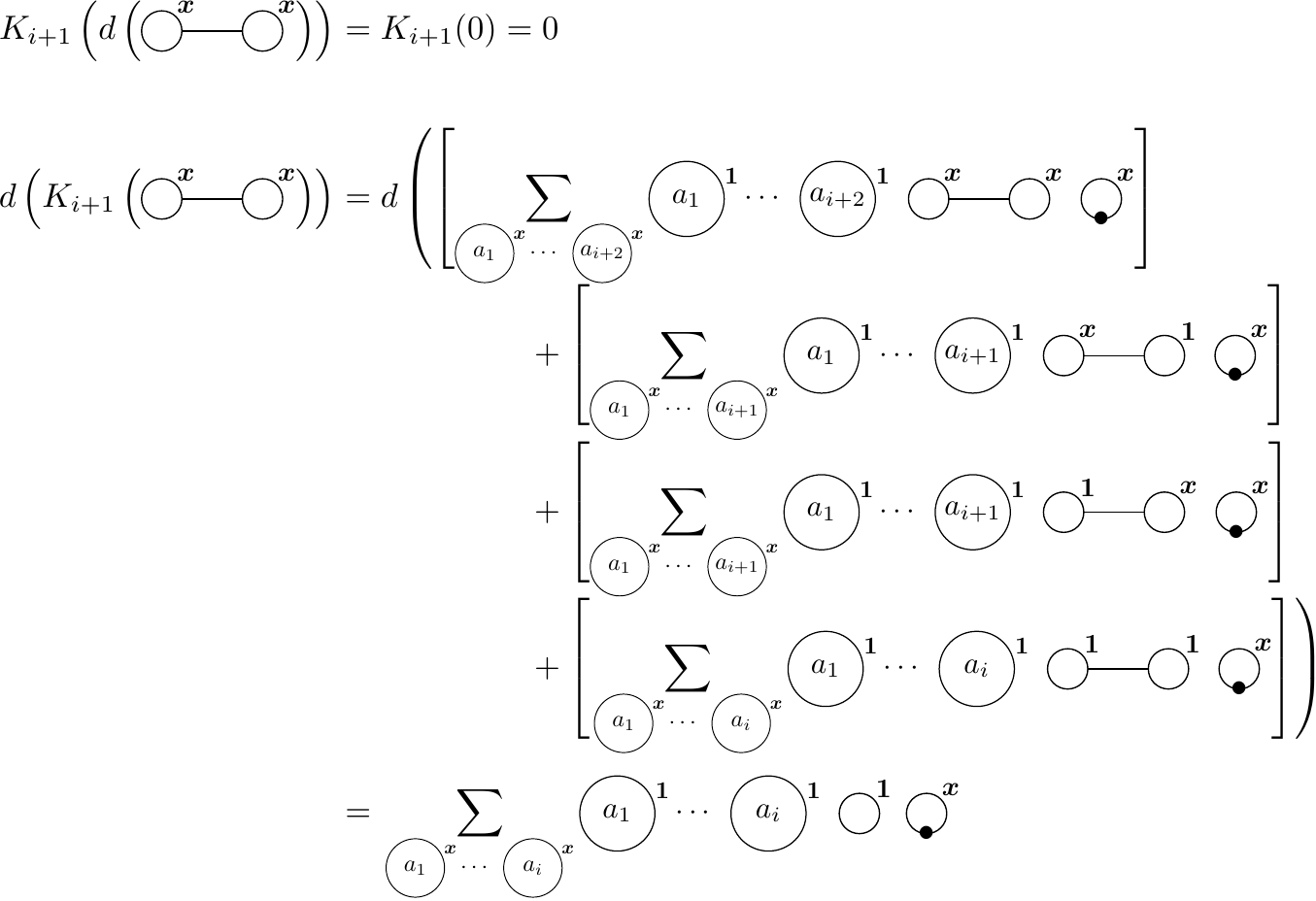}
\end{center}
This precisely cancels (4). Thus, we see that
\[[K_i,h]=[K_{i+1},d]\]
on any configuration. This means that if we define $K=K_0+HK_1+H^2 K_2+\dotsb$, then we must have $f=[d+Hh,K]$, which means that the Bar-Natan theory splits as the direct sum of its two reduced versions.

\section{Isomorphism}
So far, we have proved that the Bar-Natan homology splits as a direct sum $H(\C_1)\oplus H(\C_x)$ of the two reduced theories. Moreover, we can say something stronger:
\begin{thm}
As graded chain complexes over $\F_2[H]$, $\C_1 \cong \C_x$. Therefore, the Bar-Natan homology is the direct sum of two copies of the same (that is, isomorphic) reduced theories.
\end{thm}
\begin{proof}
We wish to construct a graded $\F_2[H]$-module isomorphism of chain complexes between $\C_1$ and $\C_x$. To do this, we first recall that in the ordinary Khovanov case of \cite{shumakovitch}, this isomorphism is given by the map $I:\C_1 \to \C_x$ that simply relabels the basepoint circle from $\1$ to $\x$. This is clearly an isomorphism (as it maps the basis vectors of $\C_1$ to the basis vectors of $\C_x$). Moreover, it is straightforward to check that $I$ is a chain map for the Khovanov differential, meaning that $[I,d]=0$.

For the Bar-Natan case, we define $\iota:C_1 \to C_x$ by
\[\iota:=I+HK=I+HK_0+H^2K_1+H^3K_2+\dotsb\]
First, we show that $\iota$ is a chain map, which means that it commutes with the differential. So we want to show that $[\iota,d+Hh]=0$. However, we have that
\begin{align*}
[\iota,d+Hh]&=[I+HK,d+Hh]\\
&=[I,d+Hh]+[HK,d+Hh]\\
&=[I,d]+[I,Hh]+Hf
\end{align*}
where the last step uses what we proved above, namely that $[K,d+Hh]=f$. Since $[I,d]=0$, it suffices to prove that
\[f=[I,h]\]
We proceed as in Section \ref{k0}. For the split maps, we have
\begin{center}
\includegraphics{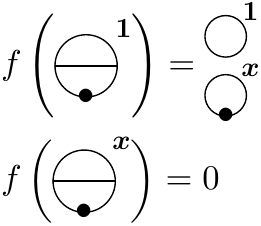}
\end{center}
and
\begin{center}
\includegraphics{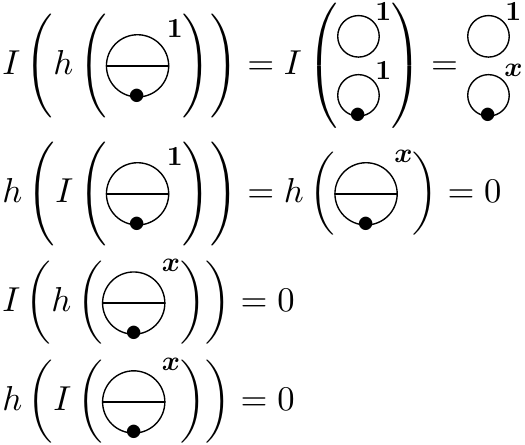}
\end{center}
so $f=[I,h]$ for the splits. For the merges, we have
\begin{center}
\includegraphics{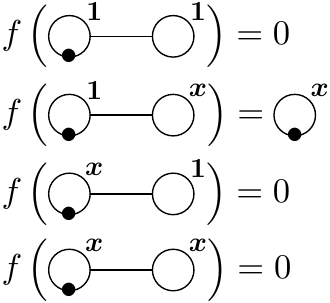}
\end{center}
For the first three of these merges, $h$ will evaluate to $0$. So all we need to check is that $h \circ I$ agrees with $f$ on them. Indeed,
\begin{center}
\includegraphics{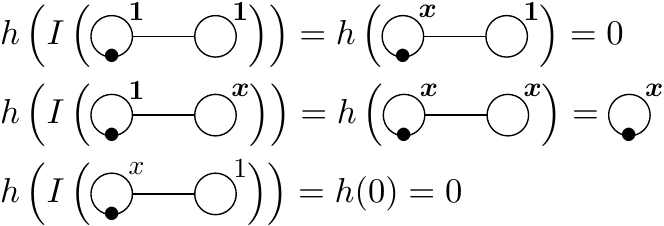}
\end{center}
For the final merge, we have that $I$ will evaluate to $0$, so we simply need to check that $I \circ h$ evaluates to $0$ as well. Indeed,
\begin{center}
\includegraphics{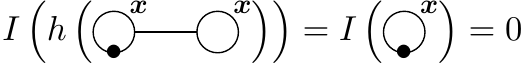}
\end{center}
Thus, we see that $f=[I,h]$. Note that unlike in the previous sections, we don't need to check that this equality holds on all surgeries not involving the basepoint circle, since $f$, $I$, and $h$ are all defined locally. So we see that the identity $f=[I,h]$ holds always, which means that $\iota$ is indeed a chain map.

Now, we need to check that $\iota$ is an isomorphism of chain complexes. We first prove that $\iota$ is injective. For suppose that we had some $\alpha \in \ker \iota$. That means that each homogeneous component of $\iota(\alpha)$ is $0$, so in particular $I(\alpha)=0$. But $I$ is injective, so we must have that $\alpha=0$. Therefore, $\iota$ is injective as well. 

Proving directly that $\iota$ is surjective is significantly more tedious, since $\iota$ is defined in terms of various sums. Luckily, we don't need to do this. For we note that since $H$ has grading $-2$, we must have that $\iota$ is a homogeneous map that always decreases quantum grading by $2$. Moreover, at each fixed quantum grading, $\C_1$ and $\C_x$ are finite-dimensional $\F_2$-vector spaces of the same dimension, since there is a bijection between their bases. So at each fixed quantum grading, $\iota$ is an injective homomorphism between two vector spaces of the same dimension, so it must be surjective as well. Since this is true for all quantum gradings, $\iota$ itself must be surjective. So $\iota$ is bijective, and as it was an $\F_2[H]$-module chain map, its inverse will be an $\F_2[H]$-module chain map as well. So $\iota$ is a graded $\F_2[H]$-module isomorphism of chain complexes, so $\C_1 \cong \C_x$. 
\end{proof}

\bibliographystyle{plain}
\bibliography{khovanovpapers}

\end{document}